\documentclass[12pt]{article}
\usepackage{amsmath,amssymb,amsbsy,amsfonts,amsthm,latexsym,
            amsopn,amstext,amsxtra,euscript,amscd}
\begin{document}
\bibliographystyle{plain}
\newfont{\teneufm}{eufm10}
\newfont{\seveneufm}{eufm7}
\newfont{\fiveeufm}{eufm5}
%
%
\newfam\eufmfam
              \textfont\eufmfam=\teneufm \scriptfont\eufmfam=\seveneufm
              \scriptscriptfont\eufmfam=\fiveeufm
\def\bbbr{{\rm I\!R}}
\def\bbbm{{\rm I\!M}}
\def\bbbn{{\rm I\!N}}
\def\bbbf{{\rm I\!F}}
\def\bbbh{{\rm I\!H}}
\def\bbbk{{\rm I\!K}}
\def\bbbp{{\rm I\!P}}
\def\bbbone{{\mathchoice {\rm 1\mskip-4mu l} {\rm 1\mskip-4mu l}
{\rm 1\mskip-4.5mu l} {\rm 1\mskip-5mu l}}}
\def\bbbc{{\mathchoice {\setbox0=\hbox{$\displaystyle\rm C$}\hbox{\hbox
to0pt{\kern0.4\wd0\vrule height0.9\ht0\hss}\box0}}
{\setbox0=\hbox{$\textstyle\rm C$}\hbox{\hbox
to0pt{\kern0.4\wd0\vrule height0.9\ht0\hss}\box0}}
{\setbox0=\hbox{$\scriptstyle\rm C$}\hbox{\hbox
to0pt{\kern0.4\wd0\vrule height0.9\ht0\hss}\box0}}
{\setbox0=\hbox{$\scriptscriptstyle\rm C$}\hbox{\hbox
to0pt{\kern0.4\wd0\vrule height0.9\ht0\hss}\box0}}}}
\def\bbbq{{\mathchoice {\setbox0=\hbox{$\displaystyle\rm
Q$}\hbox{\raise
0.15\ht0\hbox to0pt{\kern0.4\wd0\vrule height0.8\ht0\hss}\box0}}
{\setbox0=\hbox{$\textstyle\rm Q$}\hbox{\raise
0.15\ht0\hbox to0pt{\kern0.4\wd0\vrule height0.8\ht0\hss}\box0}}
{\setbox0=\hbox{$\scriptstyle\rm Q$}\hbox{\raise
0.15\ht0\hbox to0pt{\kern0.4\wd0\vrule height0.7\ht0\hss}\box0}}
{\setbox0=\hbox{$\scriptscriptstyle\rm Q$}\hbox{\raise
0.15\ht0\hbox to0pt{\kern0.4\wd0\vrule height0.7\ht0\hss}\box0}}}}
\def\bbbt{{\mathchoice {\setbox0=\hbox{$\displaystyle\rm
T$}\hbox{\hbox to0pt{\kern0.3\wd0\vrule height0.9\ht0\hss}\box0}}
{\setbox0=\hbox{$\textstyle\rm T$}\hbox{\hbox
to0pt{\kern0.3\wd0\vrule height0.9\ht0\hss}\box0}}
{\setbox0=\hbox{$\scriptstyle\rm T$}\hbox{\hbox
to0pt{\kern0.3\wd0\vrule height0.9\ht0\hss}\box0}}
{\setbox0=\hbox{$\scriptscriptstyle\rm T$}\hbox{\hbox
to0pt{\kern0.3\wd0\vrule height0.9\ht0\hss}\box0}}}}
\def\bbbs{{\mathchoice
{\setbox0=\hbox{$\displaystyle     \rm S$}\hbox{\raise0.5\ht0\hbox
to0pt{\kern0.35\wd0\vrule height0.45\ht0\hss}\hbox
to0pt{\kern0.55\wd0\vrule height0.5\ht0\hss}\box0}}
{\setbox0=\hbox{$\textstyle        \rm S$}\hbox{\raise0.5\ht0\hbox
to0pt{\kern0.35\wd0\vrule height0.45\ht0\hss}\hbox
to0pt{\kern0.55\wd0\vrule height0.5\ht0\hss}\box0}}
{\setbox0=\hbox{$\scriptstyle      \rm S$}\hbox{\raise0.5\ht0\hbox
to0pt{\kern0.35\wd0\vrule height0.45\ht0\hss}\raise0.05\ht0\hbox
to0pt{\kern0.5\wd0\vrule height0.45\ht0\hss}\box0}}
{\setbox0=\hbox{$\scriptscriptstyle\rm S$}\hbox{\raise0.5\ht0\hbox
to0pt{\kern0.4\wd0\vrule height0.45\ht0\hss}\raise0.05\ht0\hbox
to0pt{\kern0.55\wd0\vrule height0.45\ht0\hss}\box0}}}}
\def\bbbz{{\mathchoice {\hbox{$\sf\textstyle Z\kern-0.4em Z$}}
{\hbox{$\sf\textstyle Z\kern-0.4em Z$}}
{\hbox{$\sf\scriptstyle Z\kern-0.3em Z$}}
{\hbox{$\sf\scriptscriptstyle Z\kern-0.2em Z$}}}}
\def\ts{\thinspace}

\newtheorem{theorem}{Theorem}
\newtheorem{lemma}[theorem]{Lemma}
\newtheorem{claim}[theorem]{Claim}
\newtheorem{cor}[theorem]{Corollary}
\newtheorem{prop}[theorem]{Proposition}
\newtheorem{definition}[theorem]{Definition}
\newtheorem{remark}[theorem]{Remark}
\newtheorem{question}[theorem]{Open Question}

\def\qed{\ifmmode
\squareforqed\else{\unskip\nobreak\hfil
\penalty50\hskip1em\null\nobreak\hfil\squareforqed
\parfillskip=0pt\finalhyphendemerits=0\endgraf}\fi}

\def\squareforqed{\hbox{\rlap{$\sqcap$}$\sqcup$}}

\def \A {{\mathbb A}}
\def \C {{\mathbb C}}
\def \F {{\mathbb F}}
\def \L {{\mathbb L}}
\def \K {{\mathbb K}}
\def \Q {{\mathbb Q}}
\def \Z {{\mathbb Z}}
\def\cA{{\mathcal A}}
\def\cB{{\mathcal B}}
\def\cC{{\mathcal C}}
\def\cD{{\mathcal D}}
\def\cE{{\mathcal E}}
\def\cF{{\mathcal F}}
\def\cG{{\mathcal G}}
\def\cH{{\mathcal H}}
\def\cI{{\mathcal I}}
\def\cJ{{\mathcal J}}
\def\cK{{\mathcal K}}
\def\cL{{\mathcal L}}
\def\cM{{\mathcal M}}
\def\cN{{\mathcal N}}
\def\cO{{\mathcal O}}
\def\cP{{\mathcal P}}
\def\cQ{{\mathcal Q}}
\def\cR{{\mathcal R}}
\def\cS{{\mathcal S}}
\def\cT{{\mathcal T}}
\def\cU{{\mathcal U}}
\def\cV{{\mathcal V}}
\def\cW{{\mathcal W}}
\def\cX{{\mathcal X}}
\def\cY{{\mathcal Y}}
\def\cZ{{\mathcal Z}}
\newcommand{\rmod}[1]{\: \mbox{mod}\: #1}

\def\tcN{\cN^\mathbf{c}}
\def\F{\mathbb F}
\def\Tr{\operatorname{Tr}}
\def\mand{\qquad \mbox{and} \qquad}
\renewcommand{\vec}[1]{\mathbf{#1}}
\def\eqref#1{(\ref{#1})}
\newcommand{\ignore}[1]{}
\hyphenation{re-pub-lished}
\parskip 1.5 mm
\def\lln{{\mathrm Lnln}}
\def\Res{\mathrm{Res}\,}
\def\lcm{\mathrm{lcm}\,}
\def\F{{\bbbf}}
\def\Fp{\F_p}
\def\fp{\Fp^*}
\def\Fq{\F_q}
\def\ff{\F_2}
\def\ffn{\F_{2^n}}
\def\K{{\bbbk}}
\def \Z{{\bbbz}}
\def \N{{\bbbn}}
\def\Q{{\bbbq}}
\def \R{{\bbbr}}
\def \P{{\bbbp}}
\def\Zm{\Z_m}
\def \Um{{\mathcal U}_m}
\def \Bf{\frak B}
\def\Km{\cK_\mu}
\def\va {{\mathbf a}}
\def \vb {{\mathbf b}}
\def \vc {{\mathbf c}}
\def\vx{{\mathbf x}}
\def \vr {{\mathbf r}}
\def \vv {{\mathbf v}}
\def\vu{{\mathbf u}}
\def \vw{{\mathbf w}}
\def \vz {{\mathbfz}}
\def\\{\cr}
\def\({\left(}
\def\){\right)}
\def\fl#1{\left\lfloor#1\right\rfloor}
\def\rf#1{\left\lceil#1\right\rceil}
\def\AST{\mathcal{A}_{\mathrm{ST}}}
\def\AU{\mathcal{A}_{\mathrm{U}}}

\newcommand{\floor}[1]{\lfloor {#1} \rfloor}

\newcommand{\comm}[1]{\marginpar{
\vskip-\baselineskip 
\raggedright\footnotesize
\itshape\hrule\smallskip#1\par\smallskip\hrule}}
\def\rem{{\mathrm{\,rem\,}}}
\def\dist {{\mathrm{\,dist\,}}}
\def\etal{{\it et al.}}
\def\ie{{\it i.e. }}
\def\veps{{\varepsilon}}
\def\eps{{\eta}}
\def\ind#1{{\mathrm {ind}}\,#1}
               \def \MSB{{\mathrm{MSB}}}
\newcommand{\abs}[1]{\left| #1 \right|}

\title{On the Typical Size and Cancelations Among the Coefficients of
Some Modular Forms}

\author{{\sc Florian~Luca}\\ 
{School of Mathematics, University of the Witwatersrand}\\ 
{P. O. Box Wits 2050, South Africa}\\
{and}\\
{Mathematical Institute UNAM, Juriquilla}\\ 
{76230 Santiago de Quer\'etaro, M\'exico}\\
{fluca@matmor.unam.mx}\\
\and
{\sc Maksym Radziwi\l\l} \\
Centres de Recherches Mathematiques \\ Universit{\'e} de Montr{\'e}al, P. O. Box 6128\\
Montreal, QC, H3C 3J7, Canada \\
{\tt radziwill@crm.umontreal.ca}
\and
{\sc Igor E. Shparlinski}  \\
Department of Pure Mathematics\\
University of New South Wales\\
Sydney, NSW 2052, Australia\\
{\tt igor.shparlinski@unsw.edu.au}
}

\date{\today}
\pagenumbering{arabic}

\maketitle

\newpage
\begin{abstract}
We obtain a nontrivial upper bound for almost all
elements of the sequences of real numbers which are multiplicative
and at the prime indices are distributed according
to the Sato--Tate density. Examples of such sequences come from coefficients of several $L$-functions 
of elliptic curves and modular forms. In particular, we show that $|\tau(n)|\le n^{11/2} (\log n)^{-1/2+o(1)}$ for 
a set of $n$ of asymptotic density 1, where $\tau(n)$ is the Ramanujan $\tau$ function while the standard argument yields $\log 2$ instead of $-1/2$ in the power of the logarithm.
Another consequence of our result is that in the number of representations 
of $n$ by a binary quadratic form one has slightly more than square-root cancellations for almost all integers $n$.
  
 In
addition we obtain a central limit theorem for such sequences, assuming 
a weak hypothesis on the rate of convergence to the Sato--Tate law. 
For Fourier coefficients of
primitive holomorphic cusp forms such a hypothesis is known conditionally assuming the automorphy of all symmetric powers of the form and seems to be within reach unconditionally using the currently established potential automorphy.
 \end{abstract}

\section{Introduction}

\subsection{Background and motivation}

Let $\AST$ be the class of infinite sequences $\{a_n\}_{n\ge 1}$ of real numbers, which satisfy
the following properties:
\begin{itemize}
\item {\it Multiplicativity:\/} for any coprime positive integers $m,n$ 
we have $a_{mn} = a_m a_n$.
\item {\it Sato--Tate distribution:\/} for any prime $p$ we have 
$a_p \in [-2,2]$
for the angles $\vartheta_p\in [0,\pi)$ defined by
$a_p = 2 \cos \vartheta_p$, and any $\alpha \in [0, \pi)$ we
have
$$
\frac{\#\{p \le x~:~p~\mathrm{prime}, \ \vartheta_p \in [0,\alpha]\}}{\pi(x)}
\to \frac{2}{\pi} \int_0^{\alpha} \sin^2 \vartheta\, d\vartheta,
\quad \text{as}\ x\to \infty,
$$
where, as usual, $\pi(x)$ denotes the prime counting function of all primes $p\le x$.
\item {\it Growth on prime powers:\/} There exist a  
constant $\varrho>0$ such that for any integer $a\ge 2$ and 
prime $p$ we have $|a_{p^{a}}| \le p^{(a-1)/2 - \varrho}$. 
\end{itemize}

Very often such 
sequences come in the form $\lambda(n)/n^{\gamma}$ 
with some real $\gamma > 0$, 
where  
$\{\lambda(n)\}_{n \ge 1}$ is the sequence of coefficients of a Dirichlet series of a
certain
$L$-function or of Fourier  coefficients of a
certain modular form. 

In particular, we recall the 
striking results of Barnet-Lamb,  Geraghty,
Harris, and Taylor~\cite{B-LGHT}, 
Clozel,  Harris and Taylor~\cite{CHT}, 
Harris,  Shepherd-Barron and  Taylor~\cite{HS-BT},  
who established that Fourier  coefficients of several types of modular forms,
after an appropriate normalisation, belong to the class $\AST$.

The two most famous examples of  such sequences to which the results 
of~\cite{B-LGHT,CHT,HS-BT} apply are:
\begin{itemize}
\item {\it Elliptic curves\/} and are given by 
$t_n(E)/n^{1/2}$, where $\{t_n(E)\}_{n \ge 1}$ are the coefficients
of the $L$-function of an elliptic curve $E$ 
(see~\cite[Section~8.1]{Iwan} or~\cite[Section~14.4]{IwKow}).

\item {\it Ramanujan $\tau$-function\/} and are given by 
$\tau(n)/n^{11/2}$, where $\tau(n)$ is the Ramanujan $\tau$-function 
(see~\cite{Ra}).
\end{itemize}

Here, we obtain a nontrivial upper bound on the size of $|a_n|$ that 
holds for almost all $n$ and we also show that there are 
nontrivial cancellations in the summatory function of 
 $\{a_n\}_{n\ge 1} \in \AST$.  
 
We note that our results can be viewed as analogues of those 
of Fouvry and Michel~\cite{FoMic1,FoMic2} concerning Kloosterman
sums. However the approaches of~\cite{FoMic1,FoMic2}  and of our 
paper are very different. In particular, Kloosterman sums do not 
satisfy the multiplicativity condition in the definition of the 
class $\AST$.  

 \subsection{Notation}
 
 Throughout the paper, the implied constants in the symbols `$O$',  `$\ll$' 
and `$\gg$'
may occasionally, where obvious, depend on the real parameters $A$ and $\varrho$ 
and are absolute, otherwise.
We recall that the notations $U = O(V)$,  $U \ll V$  and $V \gg U$ are all
equivalent to the assertion that the inequality $|U|\le c|V|$ holds for some
constant $c>0$.

We always use $n$ to denote a positive integer and
use $p,~q$ to denote primes. We write $\log x$ for the natural logarithm of $x\ge 1$. We put $\log_1 x=\max\{\log x,1\}$ and for $k\ge 2$, we define $\log_k x$ as $\log_k x=\log_1 \log_{k-1} x$. 
Thus, for large $x$, $\log_k x$ coincides with the $k$ fold iterated composition of the natural logarithm function evaluated in $x$ (and it is $1$ for smaller values of $x$).

\subsection{Main Results}

We say that some property holds for {\it almost all\/} $n$ if the
number of $n \in [1,x]$ for which it fails is $o(x)$ as $x \to 
\infty$.  

 \begin{theorem}
 \label{thm:AST}
 For any sequence $\{a_n\}_{n \ge 1} \in \AST$,  the inequality
 $$
 |a_n| \le (\log n)^{-1/2+o(1)}
 $$
holds for almost all positive integers $n$. 
 \end{theorem}
It is certainly possible to construct artificial sequences $a_n \in \AST$ such that $a_n = 0$ for almost all $n$. However, if $a_n = t_n(E)$ and $E$ is a non-CM elliptic curve or $a_n = \tau(n)$, then the exponent $1/2$
in  Theorem~\ref{thm:AST}  is optimal (see Theorem~\ref{thm:CLT} below).

It is well-known the number
of representations of $n$ in terms of {\it a binary quadratic form\/} 
can be expressed  via coefficients of some  cusp forms, 
see~\cite[Chapter~11]{Iwan} and~\cite[Chapter~1]{Sarn}, 
Thus, one interesting consequence of our Theorem~\ref{thm:AST} is that in the number
of such representations has slightly more than square-root cancellation in the error term 
for almost all $n$ (however making a precise general statements requires imposing 
a series of technical conditions and maybe quite cluttered). 

With regards to  
the cancellations and sign changes, 
we prove the following result.
 
 \begin{theorem}
 \label{thm:AST aver}
 For any sequence $\{a_n\}_{n \ge 1} \in \AST$,  the estimate
 $$
 \sum_{n\le x} a_n=o\(\sum_{n\le x} |a_n|\)
 $$
 holds as $x\to\infty$. 
 \end{theorem}
 
As we have mentioned before for many interesting
sequences
$\{a_n\}_{n \ge 1} \in \AST$ our result is (conjecturally) optimal. Indeed, let us make the following
two additional assumptions:
\begin{itemize}
\item[(A1)] If $a_p \neq 0$ then $|a_{p^k}| \gg p^{-Ck}$ for some $C > 0$ and
all integer $k \ge 1$. 
\item[(A2)] We have, for any fixed $A > 0$,
$$\frac{
\# \{ p \leq x ~:~\vartheta_p \in [\alpha,\beta] \}}{\pi(x)} = \frac{2}{\pi} \int_{\alpha}^{\beta}
\sin^2 \vartheta d\vartheta + O\( (\log_2 x)^{-A} \).
$$
uniformly in $\pi \geq \beta \geq \alpha \geq 0$ as $x \to \infty$. 
\end{itemize}
In the case of $\tau(n)/n^{11/2}$, Assumption~A1 is known (see~\cite[Lemma~9]{Elliott2}), while  
Assumption~A2 follows from the automorphy
of $L(\text{Sym}^k f, s)$ for $f(n) = a_n$ and every $k=1,2, \ldots$, 
(see~\cite{Thorner2}). 
It would be interesting to determine whether the currently known potential automorphy is enough in order to deduce  Assumption~A2. 

For general coefficients of holomorphic
cusp forms one can prove Assumption~A1 for all primes
with at most finitely many exceptions by combining~\cite[Lemma~9]{RadziwillMatomaki} 
and~\cite[Lemma~2.2]{KowalskiRobert}. 
For coefficients of elliptic curves Assumption~A2 can be proven for fixed $\alpha,\beta$, on average 
for most elliptic curves (see~\cite{BaSh,Shp1,Shp2} for the currently strongest 
forms of this statement).

\begin{theorem}
\label{thm:CLT}
Assume that for $a_n \in \AST$  both Assumptions~A1 and~A2 hold. Let $\cN(x)=\{n\in [1,x]:a_n \neq 0\}$. Then, for fixed $v \in \mathbb{R}$,
$$
\lim_{x\to\infty} \frac{1}{\#\cN(x)} \cdot
\# \left \{ n \in \cN(x):
\frac{\log |a_n| + \frac{1}{2} \log_2 n}{\sqrt{c \log_2 n}} \geq v \right \}= \int_{v}^{\infty} e^{-u^2/2} \frac{du}{\sqrt{2\pi}}
$$
with 
$$
c = \frac 12 + \frac{\pi^2}{12}.
$$
\end{theorem}

In the recent work,  Elliott and Kish~\cite{Kish} claim to be able to prove
Theorem~\ref{thm:CLT} under  Assumption~A2 but with a stronger rate of
convergence of the form $o((\log X)^{-3})$. It is important to reduce
the error term as much as possible in Assumption~A2. For example, our assumption (but not
that of Elliott and Kish~\cite{Kish}) allows us to
use the current conditional results towards Assumption~A2 assuming automorphy.
An unconditional effective Sato-Tate theorem would certainly have an
even weaker error term, since the currently known potential automorphy
is weaker than automorphy.

It is an interesting question whether one can establish 
unconditionally the upper bound part of Theorem~\ref{thm:CLT}. 
In the case of coefficients of half-integral weight cusp forms this has been very recently done 
unconditionally in~\cite{RadziwillSoundararajan}.

\section{Preparations}
 
\subsection{Preliminaries on Arithmetic Functions}

For the proof of Theorem~\ref{thm:AST} and Theorem~\ref{thm:AST aver}, we   need the following results. 
The first one is given by~\cite[Theorem~01]{HT} or~\cite[Lemma~9.6]{KoLu}.

\begin{lemma}
\label{lem:f}
Let $f$ be a multiplicative function such that 
 $f(n)\ge 0$ for every $n$. 
Assume that there exist positive constants $A$ and $B$ such that for $x> 1$ 
both inequalities 
$$
\sum_{p\le x} f(p)\log p\le Ax\quad {\text{ and}}\quad \sum_{p} \sum_{\alpha\ge 2} \frac{f(p^{\alpha})}{p^{\alpha}} \log (p^{\alpha})\le B
$$
hold. Then
$$
\sum_{n\le x} f(n)\le (A+B+1) \frac{x}{\log x} \sum_{n\le x} \frac{f(n)}{n}.
$$
\end{lemma}

We are now ready to prove some estimates that are of 
independent interest:

\begin{lemma}
\label{lem:3}
We have
\begin{itemize}
\item[(i)] 
$$
\sum_{n\le x} \frac{|a_n|}{n}\ll (\log x)^{0.85},
$$
\item[(ii)]
$$
\sum_{n\le x} |a_n|^2\le x(\log x)^{o(1)},
$$ 
\item[(ii)] for $\gamma < 1$, 
$$
\sum_{n\le x} |a_n|^{\gamma} \le x (\log x)^{- \gamma/2 + c \gamma^2}
$$
with some absolute constant  $c > 0 $, 
\end{itemize}
as $x\to\infty$. 
\end{lemma}

\begin{proof}
To prove~(i), we note that by multiplicativity 
and by the bound on the values of $|a_n|$ at prime powers,  
we have for any fixed $\gamma\in [0,2]$, 
\begin{equation*}
\begin{split}
\sum_{n\le x}\frac{|a_n|^{\gamma}}{n} & \le  
 \prod_{p\le x} \(1+\frac{|a_p|^{\gamma}}{p}+\frac{|a_{p^2}|^{\gamma}}{p^{2}}+\cdots\)\\
& =   \prod_{p\le x} \(1+\frac{(2|\cos \vartheta_p|)^{\gamma}}{p}+O\(\frac{1}{p^{1+2\varrho}}\)\)\\
& =   \exp\(\sum_{p\le x} \frac{(2|\cos \vartheta_p|)^{\gamma}}{p} +O\(\sum_{p\ge 2} \frac{1}{p^{1+2\varrho}}\)\)\\
& =  
\exp\(\sum_{p\le x} \frac{(2|\cos \vartheta_p|)^{\gamma}}{p} +O\(1\)\).
\end{split}
\end{equation*}
Recalling  the Sato--Tate distribution property of the sequence $\{a_n\}_{n \ge 1}$, 
and the Mertens formula (see~\cite[Chapter~I.1, Theorem~9]{Ten}), via partial summation 
we obtain
$$
\sum_{p\le x} \frac{(2|\cos \vartheta_p|)^{\gamma}}{p}=\(\frac{2}{\pi}\int_{0}^{\pi} (2|\cos \vartheta|)^{\gamma} \sin^2 \vartheta d\vartheta+o(1)\)\log_2 x
$$
as $x\to\infty$. Let
$$
h(\gamma) = \frac{2}{\pi}\int_{0}^{\pi} (2|\cos \vartheta|)^{\gamma} \sin^2 \vartheta d\vartheta.
$$
The above integral evaluates to $0.848826 \ldots$ when $\gamma = 1$; hence, we obtain (i). 
The integral also evaluates to $1$ when $\gamma = 2$. Furthermore for small $\gamma < 1$
since $h'(0) = - 1/2$,
$$
h(\gamma) \leq  1 - \frac{\gamma}{2} + O(\gamma^2). 
$$
It follows that,
\begin{equation}
\begin{split}
\label{eq:13}
\sum_{n \le x} \frac{|a_n|^2}{n}   \le (\log x)^{1 + o(1)} \mand 
\sum_{n \le x} \frac{|a_n|^{\gamma}}{n}   \le (\log x)^{1 - \frac{\gamma}{2} + c \gamma^2}
\end{split}
\end{equation}
with $c > 0$ an absolute constant. 

Noting that the functions $f(n)=|a_n|^2\le 1$ and $f(n)=|a_n|^{\gamma}$ satisfy the conditions of Lemma~\ref{lem:f} 
with some appropriate constants $A$ and $B$, we derive~(ii) and~(iii) from the inequalities~\eqref{eq:13}.
\end{proof}

\section{Proof of Theorem~\ref{thm:AST}}

Let $\varepsilon > 0$ be fixed.
Let $\cN_{\varepsilon}(x)=\#\{n \leq x:|a_n| > (\log n)^{-1/2 + \varepsilon}\}$.
By Lemma~\ref{lem:3}~(iii),
$$
\cN_{\varepsilon}(x) \leq \sum_{n \leq x} |a_n|^{\gamma} \cdot (\log n)^{\gamma/2 - \gamma \varepsilon}
\ll x (\log x)^{-\gamma \varepsilon + c \gamma^2},
$$ 
where $c > 1$ is an absolute constant. Choosing $\gamma = \varepsilon / (2c)$, we get
$$
\cN_{\varepsilon}(x) \leq x (\log x)^{-\varepsilon^2 /2} = o(x)\qquad (x\to\infty).
$$
Hence, for any fixed $\varepsilon > 0$, we have $|a_n| < (\log n)^{-1/2 + \varepsilon}$ for almost
all $n$. 

\section{Proof of Theorem~\ref{thm:AST aver}}

\subsection{Sums $S$ and $T$ and negligible sets}

We let $x$ be  sufficiently large and define
$$
S=\sum_{x/2< n\le x} a_n\quad {\text{\rm and}}\quad T=\sum_{x/2<n\le x} |a_n|.
$$
It suffices to show that $S=o(T)$ as $x\to\infty$. Note first that by considering only those $n$ which are primes $p$ such that $|\cos \vartheta_p|\ge 1/2$, a set 
of density $(4/\pi)\int_0^{\pi/3} \sin^2 \vartheta d\vartheta>0.39$, we already get that
\begin{equation}
\label{eq:T}
T\ge \sum_{\substack{ x/2<p\le x\\ |\cos \vartheta_p|\ge 1/2}} 2|\cos \vartheta_p|\gg \pi(x).
\end{equation}
Suppose next that $\cN\subseteq [x/2,x]$ is such that $\#\cN=O(x/(\log x)^{3})$. Then, by
the Cauchy-Schwarz inequality  and  Lemma~\ref{lem:3}~(ii), we have,
$$
|\sum_{n\in \cN} a_n|\le \sum_{n\in \cN} |a_n|\le \(\# \cN\)^{1/2} \(\sum_{n\in \cN} |a_n|^2\)^{1/2}\le 
\frac{x}{(\log x)^{1.5+o(1)}}=o(T)
$$
as $x\to\infty$. Hence, we may neglect the contribution to either 
$S$ or $T$ coming from  any subsets $\cN$ of $[x/2,x]$ of cardinality of order at most 
$x/(\log x)^3$.

\subsection{Discarding contributions from exceptional integers}

Let 
$$
y=\exp\(\frac{4\log x\log_3 x}{\log_2 x}\).
$$
For a positive integer $m$ put $P(m)$ for the largest prime factor of $m$ with $P(1)=1$. Let
$$
\cN_1(x)=\{x/2<n\le x~:~P(n)\le y\}.
$$
{From} the theory of smooth numbers~\cite{CEP}, we know that in our range for $y$ versus $x$, 
$$
\# \cN_1(x)\ll x\exp(-(1+o(1)) u\log u),\quad {\text{\rm where}}\quad u=\frac{\log x}{\log y}\qquad (x\to\infty).
$$
Since $u=4\log_2 x/(\log_3 x)$, it follows that $u\log u=(4+o(1))\log_2 x$ as $x\to\infty$, and therefore
$$
\#\cN_1(x) \le \frac{x}{(\log x)^{3}}.
$$
{From} now on, we discard the positive integers $n\in \cN_1(x)$. Next let 
$$
\cN_2(x) =\{n\in [x/2,x]~:~p^2\mid n~{\text{\rm for~some}}~p>y/2\}.
$$
Fixing $p$, the number of $n\in [x/2,x]$ which are multiples of $p^2$ is at most $\lfloor x/(2p^2)\rfloor +1$. Thus,
$$
\#\cN_2(x)\le \sum_{y/2\le p\le x^{1/2}} \(\fl{x/(2p^2)} +1\)\ll x\sum_{y/2\le p} \frac{1}{p^2}+\pi({\sqrt{x}})\ll \frac{x}{y}
\ll \frac{x}{(\log x)^3}.
$$
{From} now on, we also discard the positive integers  $n\in \cN_2(x)$.

Consider $n\in [x/2,x]\backslash \(\cN_1(x)\cup \cN_2(x)\)$. 

We write $n=P(n)m$. Then $P(n)>\max\{y,P(m)\}$ is prime and $P(n)\in [x/(2m),x/m]$. Further, any prime 
$P$ in the interval $[x/(2m),x/m]$ can serve the role of $P=P(n)$ for $n=Pm$, except when it is the case that $P(m)\in [x/(2m),x/m]$. Let 
$\cN_3(x)$ be the set of $n$ of the form $n=P(n)m$, where $ P(n)>P(m)$ and $P(m)\in [x/(2m),x/m]$. Then $P(m)>x/(2m)\ge P(n)/2>y/2$, so 
that $P(m)\| m$ (that is, $P(m)^2 \nmid m$). Write $Q=P(m)$ and $m=Q\ell$. 
Then $n=PQ\ell$, $P>Q>P/2$ and $P(\ell)<Q$. Further, 
$$
\frac{x}{2\ell}<PQ<\frac{x}{\ell},\quad {\text{\rm therefore}} \quad \(\frac{x}{2\ell}\)^{1/2}<P<2Q<2\(\frac{x}{\ell}\)^{1/2}.
$$
Let $\cL(x)$ be the set of all possible values $\ell$  and for a fixed $\ell\in\cL(x)$,  
let $\cN_{3,\ell}(x)$ be the subset of $n = PQ\ell$  from $\cN_3(x)$ with the corresponding $\ell$.
Note that $P$, $Q$ and $\ell$ are pairwise relatively prime.
Then
\begin{equation*}
\begin{split}
\left|\sum_{n\in \cN_{3,\ell}(x)} a_n\right| & \le  
 |a_{\ell}|
\sum_{P \in \({\sqrt{x/(2\ell)}}, 2{\sqrt{x/\ell}}\)} 
\sum_{Q \in \({\sqrt{x/(8\ell)}}, {\sqrt{x/\ell}}\)} 
|a_P| |a_Q|\\
& \ll   |a_{\ell}|\pi\(2 \(\frac{x}{\ell}\)^{1/2}\)^2\ll   \frac{x|a_{\ell}|}{\ell(\log(x/\ell))^2} \ll   \frac{x}{(\log y)^2} \frac{|a_{\ell}|}{\ell}.
\end{split}
\end{equation*}
Summing up over all the possible values of $\ell$ and invoking Lemma~\ref{lem:3}~(i) and estimate~\eqref{eq:T}, we get that
\begin{equation*}
\begin{split}
\left|\sum_{n\in \cN_3(x)} a_n\right| & \le    \sum_{\ell\in\cN(x)} \sum_{n\in \cN_{3,\ell}(x)(x)} |a_n|  \ll   \frac{x}{(\log y)^2} \sum_{\ell\le x} \frac{|a_{\ell}|}{\ell}\\
& \ll  \frac{x(\log x)^{0.85} (\log_2 x)^2}{(\log x)^2 (\log_3 x)^2}\ll \frac{x}{(\log x)^{1.1}}=o(T)
\end{split}
\end{equation*}
as $x\to\infty$.

\subsection{Contributions to $S$ and $T$ from typical integers}

Define 
$$
\cN(x) \in  [x/2,x]\backslash \(\cN_1(x)\cup \cN_2(x)\cup \cN_3(x)\).
$$
For each such $n \in \cN(x)$, we write as before $n=Pm$
and let $\cM(x)$ be the set of all possible values of $m$. Fix $m\in \cM(x)$. We compare
$$
S_m=\sum_{\substack{n \in \cN(x)\\n = P(n)m}} a_n=a_m\sum_{P\in [x/2m,x/m]} a_P
$$
and 
$$
T_m=\sum_{\substack{n \in \cN(x)\\n = P(n)m}}
|a_n|=|a_m|\sum_{P\in [x/(2m),x/m]} |a_P|.
$$
Note that
$$
 \int_0^{\pi} \cos \vartheta \sin^2 \vartheta d\vartheta = 0
 \mand
 \int_0^{\pi} |\cos \vartheta| \sin^2 \vartheta d\vartheta= \frac{1}{3}.
 $$
 Clearly, since $\{a_n\}_{n \ge 1} \in \AST$, we have
\begin{equation}
\begin{split}
\label{eq:S1}
\sum_{P\in [x/(2m),x/m]}& a_P  =  2\sum_{P\in [x/(2m),x/m]} \cos \vartheta_P\\
&=\(\frac{4}{\pi} \int_0^{\pi} \cos \vartheta \sin^2 \vartheta d\vartheta+o(1)\)\(\pi(x/m)-\pi(x/2m)\)\\
&=o(\pi(x/m)),
\end{split}
\end{equation}
whereas
\begin{equation}
\begin{split}
\label{eq:T1}
\sum_{P\in [x/(2m),x/m]}& |a_P|   =  2\sum_{P\in [x/(2m),x/m]} |\cos \vartheta_P|\\
&=\(\frac{4}{\pi} \int_0^{\pi} |\cos \vartheta| \sin^2 \vartheta d\vartheta+o(1)\) \(\pi(x/m)-\pi(x/2m)\)\\
& \gg  \pi(x/m).
\end{split}
\end{equation}
We see from the definition of $S_m$ and $T_m$ that if
 $a_m = 0$ then $S_m = T_m = 0$.
Comparing~\eqref{eq:S1} and~\eqref{eq:T1}, we obtain that
if $a_m \ne 0$ then $S_m = o(T_m)$. Therefore,
\begin{equation*}
\begin{split}
S & =   \sum_{n\in [x/2,x]} a_n=\sum_{m\in \cM(x)} S_m+o(T) 
=\sum_{\substack{m\in \cM(x)\\ a_m \ne 0}} S_m+o(T)\\
& = \sum_{\substack{m\in \cM(x)\\ a_m \ne 0}}  o\(T_m\)+o(T) = \sum_{m\in \cM(x)} o\(T_m\)+o(T)\\
& =   o\(\sum_{m\in \cM(x)}T_m\)+o(T)= o(T)\qquad (x\to\infty),
\end{split}
\end{equation*}
as desired.

\section{Proof of Theorem~\ref{thm:CLT}}

\subsection{The sets $\cN(x)$ and $\cA(x)$} 

As in the statement of Theorem~\ref{thm:CLT}, let $\cN(x)=\{n\le x~:~a_n\ne 0\}$. 

We recall that all implied constants may depend on the parameter $A$.

\begin{lemma}
\label{eq:sum recip}
For a fixed $A > 1$, and $y\ge 2$, then
$$
\sum_{\substack{p\ge y\\ |a_p| < (\log_2  p)^{-A}}} \frac{1}{p} \ll \frac{1}{(\log_2 y)^{A-1}}.
$$
\end{lemma}

\begin{proof}
We may assume that $y\ge 10$. Let $k_y\ge 3$ be the minimal positive integer such that $e^{k_y}\ge y$. Splitting the summation range into intervals of the form 
$[e^k, e^{k + 1}]$ for integers $k\ge k_y-1$, and using Assumption~A2, we get
\begin{equation*}
\begin{split}
\sum_{\substack{p\ge y\\ |a_p| < (\log_2 p)^{-A}}} \frac{1}{p}
& \le   \sum_{k\ge k_y-1}^\infty \sum_{\substack{e^k < p \leq e^{k + 1} \\
|a_p| < (\log k)^{-A}}} \frac{1}{p}  \ll   \sum_{k\ge k_y-1}^\infty  \frac{1}{e^k} \sum_{\substack{p < e^{k + 1} \\
|a_p| < (\log k)^{-A}}} 1 \\ 
&   \ll  \sum_{k\ge k_y-1} \frac{1}{e^k} \cdot \frac{e^{k+1}}{k (\log k)^{A}} \ll  \sum_{k\ge k_y-1} \frac{1}{k(\log k)^A}\\
 & \ll   \frac{1}{(\log k_y)^{A-1}}\ll \frac{1}{(\log_2 y)^{A-1}},
\end{split}
\end{equation*}
as claimed. 
\end{proof}

Put
$$
\kappa=\prod_{\substack{p\ge 2\\ a_p=0}} \(1-\frac{1}{p}\).
$$
Lemma~\ref{eq:sum recip} with any $A>1$ shows that $\kappa>0$. 

\begin{lemma}
\label{eq:card Nx}
For a fixed $A>1$, we have 
$$
\# \cN(x) = \kappa x \(1+O\(\frac{1}{(\log_2 x)^{A-1}}\)\).
$$
\end{lemma}

\begin{proof}
Let $f$ be a multiplicative function such that $f(p^{\alpha}) = f(p) = 1$
if $a_p\neq 0$ and $f(p) = 0$ otherwise. 
Write
$$
f(n) = \sum_{d\mid n} g(d)
$$
with $g$ a multiplicative function such that $g(p^{\alpha}) = f(p) - 1$ if $\alpha=1$ and $p\le x$
and $g(p^{\alpha}) = 0$ otherwise.  Then,
\begin{equation} \label{equeasy}
\# \cN(x) = \sum_{n \leq x} f(n) = \sum_{n \leq x} \sum_{d\mid n} g(d)
= \sum_{d \leq x} g(d) \( \frac{x}{d} + O(1) \).
\end{equation}
By Lemma~\ref{lem:f}, 
$$
\sum_{n \leq x} |g(n)| \ll \frac{x}{\log x} 
\prod_{\substack{p \le x\\ p \not\in  \cN(x)}}  \(1 + \frac{1}{p} \).
$$
Since $p\not\in \cN(x)$ implies, in particular, that $|a_p|<(\log_2  p)^{-2}$, we have 
$$
\sum_{\substack{p \le x\\ p \notin  \cN(x)}}\frac{1}{p}  = O(1)
$$ 
by Lemma~\ref{eq:sum recip} with $A=2$ and $y=2$. Hence,
$$
\prod_{\substack{p\le x\\ p\not\in \cN(x)}} \(1+\frac{1}{p}\)=O(1).
$$
Furthermore, observe that
$$
\sum_{d > x} \frac{|g(d)|}{d} = \sum_{\substack{d>x\\ p\mid d \Rightarrow p\not\in \cN(x)}} \frac{\mu^2(d)}{d}. 
$$
To evaluate the last sum above, we split it at $y=\exp((\log x)^{1/3})$. In the lower range, the sum is bounded by 
$$
S_1=\sum_{\substack{d>x\\ P(d)<y}} \frac{\mu(d)^2}{d}.
$$
Putting $u=\log x/\log y=(\log x)^{2/3}$, we have, by known estimates concerning smooth numbers, 
$$
S_1\ll \frac{1}{\exp(u)} \sum_{P(n)\le y} \frac{1}{n}\ll \frac{\log y}{e^{u}}\ll \frac{1}{e^{u/2}}\ll \frac{1}{(\log_2 x)^{A-1}}.
$$
In the upper range, we have $P(d)\ge y$, so writing $d=pm$, where $p=P(d)$, we get that the sum in this range is bounded by
\begin{equation*}
\begin{split}
S_2 & =  \sum_{\substack{p\not\in {\mathcal N}(x)\\ p>y}} \frac{1}{p} \sum_{\substack{m\\ q\mid m\Rightarrow q\not\in \cN(x)}} \frac{1}{m}
 \ll  \prod_{q\not\in \cN(x)} \(1+\frac{1}{q}\) 
\sum_{\substack{p\not\in {\mathcal N}(x)\\ p>y}} \frac{1}{p} \\
&  \ll  \frac{1}{(\log_2 y)^{A-1}}\ll \frac{1}{(\log_2 x)^{A-1}},
\end{split}
\end{equation*}
by Lemma~\ref{eq:sum recip}. Hence, the equation~\eqref{equeasy} becomes
$$
\# \cN(x) = \prod_{p\le x} \(1 + \frac{g(p)}{p} \) x
\(1+O\(\frac{1}{(\log_2 x)^{A-1}}\)\).
$$
It remains to show that
$$
\kappa=\prod_{p\le x} \(1+\frac{g(p)}{p}\) 
\(1+O\(\frac{1}{(\log_2 x)^{A-1}}\)\).
$$
However this is equivalent to 
$$
\prod_{\substack{p>x\\ a_p=0}} \(1-\frac{1}{p}\)
=1+O\(\frac{1}{(\log_2 x)^{A-1}}\),
$$
and this follows immediately from Lemma~\ref{eq:sum recip} for any $A>1$ and $y=x$. 
\end{proof}

Note that by Assumption~A2, for every $A > 0$, we have
$$
\frac{1}{\pi(x)} \cdot \# \left \{p \leq x~:~|a_p| < (\log_2  x)^{-A} \right \}
\ll (\log_2  x)^{-A}.
$$
Instead of working on the set $ \cN(x)$ we  work on the more 
convenient set 
$$
\cN_A(x) = \left \{n \leq  \cN(x)~:~p\mid n \implies |a_p| > (\log_2  x)^{-A} \right \}.
$$
The following result justifies this change.

\begin{lemma}
\label{eq:Set A}
For a fixed $A> 1$, 
uniformly for sets $\cA\subset \R$,  we have
\begin{equation*}
\begin{split}
\frac{\# \left \{ n \in  \cN(x)~:~\log |a_n| \in \cA \right \} }{\# \cN(x)} & =  \frac{ \# \left \{ n \in  \cN_A(x):
\log |a_n| \in \cA \right \}}{\# \cN_A(x)} \\
& +  O\(\frac{1}{(\log_4 x)^{A-1}}\)
\end{split}
\end{equation*}
as $x\to\infty$.
\end{lemma}

\begin{proof}
Let $\cE_A(x)$ be the set of those $n \in  \cN(x)$ which are divisible by a
prime $p \leq x$ such that $|a_p| < (\log_2  x)^{-A}$. Since $a_p \neq 0$, by Assumption~A1, we have 
$|a_p| \gg p^{- C}$ for some
absolute constant $C > 0$. Hence, $|a_p| < (\log_2  x)^{-A}$ implies that $p > (\log_2 x)^{A/(2C)}$ provided that $x$ is sufficiently large.
Thus, using Lemma~\ref{eq:sum recip} with $y=(\log_2 x)^{A/(2C)}$, we have 
\begin{equation}
\label{eq:E}
\# \cE_A(x) \ll 
\sum_{\substack{p \leq x \\ p \notin  \cN_A(x)}} \frac{x}{p} \leq \sum_{\substack{p > (\log_2 x)^{A/(2C)} \\
|a_p| < (\log_2  p)^{-A}}} \frac{x}{p} \ll \frac{x}{(\log_4 x)^{A-1}}.
\end{equation}
Hence, 
$$
\# \cN(x) = \# \cN_A(x) + O\(\frac{x}{(\log_4 x)^{A-1}}\),
$$ 
and since 
$$
\# \cN(x) =\kappa x
\(1+O\(\frac{1}{(\log_2 x)^{A-1}}\)\),
$$
by Lemma~\ref{eq:card Nx}, we have that  
\begin{equation}
\label{eq:NA}
\# \cN(x) =
 \(1 + O\(\frac{1}{(\log_4 x)^{A-1}}\)\) \# \cN_A(x).
\end{equation}
In addition, 
$$
\# \{n \in  \cN(x):\log |a_n| \in \cA \} = 
\# \{n \in  \cN_A(x):\log |a_n| \in \cA \} + 
O(\# \cE_A(x)).
$$
Dividing the last relation above by $\# \cN(x)$ and using~\eqref{eq:E} and~\eqref{eq:NA}, we obtain the claim.
\end{proof}

\subsection{Approximation to $a_n$}

We now define a strongly multiplicative (that is, $h(p^k) = h(p)$) function $h$
such that $h(p) = a_p$. Next we show that $h(n)$ is a
good approximation to $a_n$. 

\begin{lemma}
\label{lem:approx}
Let $A>1$ be fixed and let  $\psi(x) \to \infty$ as $x \to \infty$. 
For all but $o(x)$ integers $n \in  \cN_A(x)$ we have
$$ \log |h(n)| = \log |a_n| + O(\psi(x)).$$
\end{lemma}
\begin{proof}
Let $c(n) = \log |h(n)| - \log |a_n|$. 
Notice that for $n \in  \cN_A(x)$ by Assumption~A1 and our standard assumption
that $|a_{p^\alpha}| < p^{(\alpha-1)/2 - \delta}$, we have
$$
|c(p^\alpha)| \ll \alpha \log p.
$$
Consider,
\begin{equation*}
\begin{split}
\sum_{n \in  \cN_A(x)} |c(n)|^2 = 
\sum_{n \in  \cN_A(x)} \left| \sum_{\substack{p^{\alpha} \| n \\
\alpha \geq 2}} c(p^{\alpha}) \right|^2 & =
\sum_{\substack{p^{\alpha}, q^{\beta} \leq x \\ \alpha, \beta \geq 2}}
|c(p^{\alpha})| \cdot |c(q^{\beta})| \sum_{\substack{n \in  \cN_A(x) \\
\lcm[p^{\alpha}, q^{\beta}] \mid n}} 1 \\
& \ll  x\sum_{\substack{p^{\alpha}, q^{\beta} \leq x \\ \alpha, \beta \geq 2}}
\frac{\alpha \log p \cdot \beta \log q}{\lcm[p^{\alpha}, q^{\beta}]} = O(x).
\end{split}
\end{equation*}
It is now obvious that for all but $O(x / \psi^2(x))$ integers
$n \leq x$ we have $|c(n)| = O(\psi(x))$. 
\end{proof}

\subsection{Some properties of the set $\cN_A(x)$}

We are almost ready to prove Theorem~\ref{thm:CLT}, but before hand we need the following
three results about the set $\cN_A(x)$. 

\begin{lemma}
\label{lem:div n}
Let $A>1$ be fixed and let $f_{A,x}$ be the multiplicative function such that $f_{A,x}(p^k) = f_{A,x}(p)$ and
$f_{A,x}(p) = 1$ if $p \in  \cN_A(x)$ and $f_{A,x}(p) = 0$ otherwise. Then,
uniformly over  integers $d \ge 1$,
$$
\sum_{\substack{d\mid n \\ n \in  \cN_A(x)}} 1
= \frac{f_{A,x}(d)}{d} \cdot \# \cN_A(x) + r_{A,x,d}, 
$$
where 
$$
|r_{A,x,d}| \ll \frac{x}{d(\log_2 (x/d))^{A-1}}.
$$ 
\end{lemma}

\begin{proof}
Notice that $f_{A,x}$ is completely multiplicative, so that 
$$
f_{A,x}(dn) = f_{A,x}(d)f_{A,x}(n).
$$ 
Thus,
$$
\sum_{\substack{d\mid n \\ n \in  \cN_A(x)}} 1 = 
\sum_{n \leq x/d} f_{A,x}(d n) = f_{A,x}(d) \sum_{n \leq x/d} f_{A,x}(n) .
$$
Write
$$
f_{A,x}(n) = \sum_{d\mid n} g_{A,x}(d),
$$ 
with $g_{A,x}(p^k) = f_{A,x}(p) - 1$ for $k=1$ and $p \leq x$
and $g_{A,x}(p^k) = 0 $ otherwise. 
Clearly,
\begin{equation}
\label{eq:d}
\sum_{n \leq x/d} f_{A,x}(n)  = \sum_{n \leq x/ d} \sum_{d\mid n} g_{A,x}(d) 
 = \sum_{e \leq x/d} g_{A,x}(e)\( \frac{x}{d e} + O(1) \).
\end{equation}
Since
\begin{equation}
\label{eq:e}
 \sum_{e=1}^\infty \frac{ g_{A,x}(e)}{e} = \prod_{p}\(1 - \frac{ g_{A,x}(p)}{p}\), 
\end{equation}
it remains to estimate the sums 
$$
\sum_{e \leq x/d} |g_{A,x}(e)| \qquad {\text{\rm and}}\qquad \sum_{n > x/d} \frac{|g_{A,x}(n)|}{n}.
$$
However, it is easy to see that
\begin{equation}
\label{eq:10}
\sum_{e \leq x/d} |g_{A,x}(e)| \ll \frac{x}{d \log (x/d)},
\end{equation}
and 
\begin{equation}
\label{eq:11}
 \sum_{n > x/d} \frac{|g_{A,x}(n)|}{n}\ll \frac{x}{d (\log_2(x/d))^{A-1}},
 \end{equation}
 by Lemmas~\ref{lem:f} and~\ref{eq:sum recip}, and the argument from the proof of Lemma~\ref{eq:card Nx}, respectively. Making $d=1$ in~\eqref{eq:d}, 
 we see that 
\begin{equation}
\begin{split}
 \label{eq:main}
\#\cN_A(x) & =   \sum_{n\in \cN_A(x)} 1= x\sum_{e=1}^\infty \frac{g_{A,x}(e)}{e} 
+O\(\frac{x}{(\log_2 x)^{A-1}}\)\\
&=  x\prod_{p}\(1 - \frac{ g_{A,x}(p)}{p}\)+O\(\frac{x}{(\log_2 x)^{A-1}}\),
\end{split}
\end{equation}
 and we derive the desired conclusion from~\eqref{eq:d}  and~\eqref{eq:e}, error estimates~\eqref{eq:10}
 and~\eqref{eq:11}, and the main term~\eqref{eq:main}. 
\end{proof}

We define
$$
\mu_{A,x} =  \sum_{p \in  \cN_A(x)} \frac{\log |a_p|}{p} \mand
\sigma_{A,x}^2 =  \sum_{p \in  \cN_A(x)} \frac{(\log |a_p|)^2}{p} 
\(1 - \frac{1}{p} \).
$$

\begin{lemma} 
\label{lem:mu sigma} 
For a fixed $A> 1$, we have
\begin{equation*}
\begin{split}
\mu_{A,x} &  = - \frac 12 \log_2 x + O(\log_3 x), \\
\sigma_{A,x}^2  & = \( \frac 12 + \frac{\pi^2}{12} \) \log_2  x + O\((\log_3 x)^2\).
\end{split}
\end{equation*}
\end{lemma}

\begin{proof}
By Lemma~\ref{eq:sum recip}, applied with $A=2$ and $y = 3$
we see that the primes $p$ with $|a_p| \le  (\log_2 x)^{-2}$ contribute 
$O(\log_3 x)$ to $\mu_{A,x}$ and $O((\log_3 x)^2)$ to $\sigma_{A,x}^2$.

The contribution of the remaining $p$  to $\mu_{A,x}$, by integration by parts, 
can be estimated as 
\begin{equation*}
\begin{split}
\frac{2\log_2 x}{\pi} & \int_{(\log_2 x)^{-2}}^{2} \log u \cdot \sqrt{1 - (u/2)^2} du   + O(\log_3 x)\\
&=\frac{2\log_2 x}{\pi} \int_{0}^{2} \log u \cdot \sqrt{1 - (u/2)^2} du   + O(\log_3 x)\\
& = - \frac{1}{2}\log_2 x + O(\log_3 x)
\end{split}
\end{equation*}
which gives the first part of the claim. 
 
The proof of the second part of the claim  uses the fact that
$$
\frac{2}{\pi} \int_{0}^{2} (\log u)^2 \sqrt{1 - (u/2)^2} du = \frac{1}{2} + \frac{\pi^2}{12},
$$
and is completely similar. 
\end{proof}

Finally using the work of Granville and Soundararajan~\cite{GranvilleSoundararajan} 
and Lemmas~\ref{lem:div n} and~\ref{lem:mu sigma}, we are able to conclude. 

\begin{lemma}
\label{lem:CLT Ax}
For a fixed $A> 1$, we have,
$$
\frac{1}{\# \cN_A(x)}  \# \left \{ n \in  \cN_A(x):
\frac{\log |g(n)| + \tfrac 12 \log_2 x}{\sqrt{c \log_2 x}}
\in [\alpha, \beta] \right\} \to \int_{\alpha}^{\beta} e^{-u^2/2} 
\frac{du}{\sqrt{2\pi}}
$$
as $x\to\infty$, with 
$$c = \frac 12 + \frac{\pi^2}{12}.
$$ 
\end{lemma}

\begin{proof}
Let $\cP$ be the set of primes $p \in  \cN_A(x)$
which are less than $y = x^{1/\log_3  x}$. Let $M(x) = \log_3  x$.
Note that $|\log |g(p)|| \ll M(x)$ for every $p \in \cP$.  
Applying a result of  Granville and 
Soundararajan~\cite[Proposition~4]{GranvilleSoundararajan}, 
in the notation of Lemmas~\ref{lem:div n} and~\ref{lem:mu sigma}, 
we get for every fixed even integer $k$,
\begin{equation*}
\begin{split}
\sum_{a \in  \cN_A(x)} \( \sum_{\substack{p\mid a \\ p \in \cP}}
\log |g(p)| - \mu_{A,y} \)^{k}&=   C_k \# \cN_A(x) \sigma_{A,y}^k \( 1
+ O \( \frac{ M(x)}{\sigma_{A,x}^2} \) \) \\
&\qquad  +  
O \( M(x)^k (\log_2 x)^k \sum_{d \in \cD_k(\cP)} |r_{A,x,d}|\), 
\end{split}
\end{equation*}
where the implied constant may depend on both $A$ and $k$ and 
$$
C_k = \frac{\Gamma(k + 1)}{2^{k/2} \Gamma(k/2 + 1)},
$$ 
is the $k$-th Gaussian moment and $\cD_k(\cP)$ is the set of integers
formed out of $k$ prime factors from the set $\cP$. For $k$ odd, we have,
\begin{equation*}
\begin{split}
\sum_{a \in  \cN_A(x)} & \( \sum_{\substack{p \mid a \\ p \in \cP}}
\log |g(p)| - \mu_y \)^k \\
& \ll \# \cN_A(x) \sigma_y^{k} \frac{k^{3/2} M(x)}{\sigma_{A,y}} + 
M(x)^k (\log_2 x)^k \sum_{d \in \cD_k(\cP)} |r_{A,x,d}|,
\end{split}
\end{equation*}
where the implied constants may also depend on $k$. 
Using Lemma~\ref{lem:div n}, we see that
$$
\sum_{d \in \cD_k(\cP)} |r_{A,x,d}| \ll
\sum_{d \in \cD_k(\cP)} \frac{x}{d (\log_2 (x^{1-k/\log_3 x}))^{3k}} \ll x (\log_2 x)^{-2k},
$$
provided $A>3k+1$. Therefore, we conclude that
$$
\sum_{a \in  \cN_A(x)} \( \sum_{\substack{p\mid a \\ p \in \cP}}
\log |g(p)| - \mu_y \)^{k} = \(\eta_k +o(1)\) C_k  \# \cN_A(x)  \sigma_{A,y}^k 
$$
as $x\to\infty$, where 
$$
\eta_k= \begin{cases}
1  & \text{ if } k \text{ even,} \\
0 &   \text{ if } k \text{ odd.}
\end{cases}
$$
 
By the method of moments the above estimates imply that
\begin{equation}
\begin{split} 
\label{methodmoments}
\frac{1}{\# \cN_A(x)} &
\# \left \{ n \in  \cN_A(x) ~:~\frac{\sum_{\substack{p\mid n\\ p \in \cP}} \log |g(p)| - \mu_{A,y}}{\sigma_{A,y}} \in [\alpha, \beta] \right \} \\
& \qquad \qquad\qquad \qquad\qquad \qquad\qquad \qquad \quad \to
\int_{\alpha}^{\beta} e^{-u^2/2} \frac{du}{\sqrt{2\pi}}
\end{split}
\end{equation}
as $x\to\infty$.  Notice that since $p \in  \cN_A(x)$ we have $|\log |g(p)|| \ll M(x)$. Hence,
for $n \in  \cN_A(x)$, 
$$
\left|\sum_{\substack{p \mid n\\ p \in \cP}} \log |g(p)| - \log |g(n)|\right| \ll M(x)\log_3 x  \ll (\log_3 x)^2. 
$$
In addition, by Lemma~\ref{lem:mu sigma}, for the above choice of $y$, 
after simple calculations, 
we derive
$$
\mu_{A,y} = - \frac 12 \log_2 x + O(\log_3 x)
$$ 
and similarly, 
$$
\sigma_{A,y}^2= \( \frac 12 + \frac{\pi^2}{12} \) \log_2 x + O( \log_3 x)^2).
$$ 
Combining this with~\eqref{methodmoments}, we obtain the claim.
\end{proof}

\subsection{Concluding the proof}
 The result follows upon combining Lemmas~\ref{eq:Set A}, \ref{lem:approx} and~\ref{lem:CLT Ax}.

\section{Comments}

Let $\tau(n)$ by the Ramanujan $\tau$ function
defined by
$$
\sum_{n=1}^{\infty} \tau(n) q^n=q\prod_{k\ge 1} (1-q^k)^{24},
$$
introduced in~\cite{Ra}.  
By the celebrated Deligne~\cite{De} bound, we can write 
$$
\tau(p)=2 p^{11/2} \cos \vartheta_p
$$
with $\vartheta_p \in [0,\pi]$. It is also known that 
$$
\tau(n)\le d(n)n^{11/2},
$$
where $d(n)$ denotes the number of divisors of $n$ 
(see~\cite[Equation~(3.32)]{Iwan}).
Recall that $d(n)\le(\log n)^{\log 2+o(1)}$  for almost all $n$
(see~\cite[Equation~(7.14) and Problem~7.3]{KoLu}
 or~\cite[Chapter~I.5, Equation~(9) and Theorem~5]{Ten}). In particular, the inequality
 $$
 \tau(n)<n^{11/2} (\log n)^{\log 2+o(1)}
 $$ 
holds for almost all $n$. 
Recalling that as $\tau(n)$ is  multiplicative (see~\cite[Equation~(3.34)]{Iwan}),
and by~\cite{B-LGHT,CHT,HS-BT}
we know that  $\tau(p)/p^{11/2}$ satisfies the Sato--Tate distribution, we 
derive from Theorem~\ref{thm:AST} that in fact the inequality
$$
 \tau(n) \le n^{11/2} (\log n)^{-0.5 + o(1)}
$$ 
holds as $n$ tends to infinity in a density $1$ subset of ${\mathbb N}$. 

It is certainly interesting to have a more quantiative version of
Theorem~\ref{thm:AST aver} with an explicit saving on the right hand 
side. Such a general result seems impossible as one needs 
concrete  estimates on the rate of convergency to the 
Sato--Tate distribution in the definition of the sequence
 $\{a_n\}_{n\ge 1}$. However, for the above concrete examples
our method can lead to such a result. 

\section{Acknowledgements}

This paper started during a visit of F.~L. to the Department of Computing of Macquarie University in  March 2013. F. L. thanks this Institution 
for the hospitality.   

During the preparation of this paper,
F.~L. was partially supported in part by a  Marcos Moshinsky fellowship;  M.~R. was partially supported by NSF Grant DMS-1128155, I.~S. was partially supported by ARC Grant DP130100237.

\end{document}